\documentclass[11pt]{article}

\usepackage{amssymb}
\usepackage{amsfonts}
\usepackage{amsmath}
\usepackage{amsthm}
\usepackage{tikz}
\usepackage{float}
\usepackage{amsmath,amssymb,amsthm,amsfonts,amsbsy, enumerate}
\usepackage{tikz}
\usepackage{float}
\usepackage[dvips]{epsfig}
\usepackage{graphicx}
\usepackage{url}
\usepackage{bbm, dsfont}

\theoremstyle{plain} 
\newtheorem{thm}{Theorem}[section]

\newtheorem{cor}[thm]{Corollary}
\newtheorem{lem}[thm]{Lemma}
\newtheorem{defn}[thm]{Definition}

\theoremstyle{remark} 
\newtheorem{rmk}[thm]{Remark}

\usepackage{graphicx}
\newcommand\independent{\protect\mathpalette{\protect\independent}{\perp}} 
\def\independent#1#2{\mathrel{\rlap{$#1#2$}\mkern2mu{#1#2}}}

\def\phi{\varphi}

\def\ben{\begin{eqnarray*}}
\def\een{\end{eqnarray*}}
\def\be{\begin{eqnarray}}
\def\ee{\end{eqnarray}}

\def\bln{\begin{align*}}
\def\eln{\end{align*}}

\begin{document}


\title{An Exact Upper Bound on the $L^p$ Lebesgue Constant and \\
The $\infty$-R\'enyi Entropy Power Inequality for Integer Valued Random Variables}
\author{
Peng Xu\thanks{Department of Mathematics, Eastern Michigan University. E-mail: {\tt pxu@emich.edu}}, 
Mokshay Madiman\thanks{Department of Mathematical Sciences, University of Delaware.
E-mail: {\tt madiman@udel.edu}}, 
James Melbourne\thanks{Department of Mathematical Sciences, University of Delaware.
E-mail: {\tt jamesm@udel.edu}.
} }
\date{}
\maketitle

\begin{abstract}
In this paper, we proved an exact asymptotically sharp upper bound of the $L^p$ Lebesgue Constant (i.e. the $L^p$ norm of Dirichlet kernel) for $p\ge 2$. As an application, we also verified the implication of a new $\infty$-R\'enyi entropy power inequality for integer valued random variables.
\end{abstract}

\section{Introduction.}
The best constants of important operators of harmonic analysis is always an area of persistent investigation, for example, the norm of the Fourier transform (FT) on locally compact abelian (LCA) groups (some basic facts of Fourier analysis on LCA groups can be found in \cite{Gf14:book, Zg59:book}). In particular, for Euclidean space, the norm of FT from $L^p(\mathbb{R}^d)$ to $L^{p'}(\mathbb{R}^d)$ with $p'$ the H\"older dual index of $p$ and $p\in (1,2]$ on Euclidean space is the content of Hausdorff-Young inequality, and the sharp constant is proven by Beckner in \cite{Bec75}. For some abstract LCA groups, the norm of FT is proven by Gilbert and Rzeszotnik in \cite{GR10} for the case that the group is finite, and by Madiman and Xu in \cite{MX16FN, XM15} for the case that the group is infinite and discrete or compact. 

Moreover, a lot of questions about estimating the $L^p$ norms of the FT of some special functions have been considered. For example, the upper bound of the $L^p$ norm of the FT of uniform probability distribution functions on intervals is proven by K. Ball in \cite{Ball86, Ball89} and by Nazarov and Podkorytov in \cite{NP} with the following sharp result:
\be\label{inq:Ball}
\mbox{Ball's~integral~inequality:~~}\int_{\mathbb{R}}\left|\frac{\sin \pi x}{\pi x} \right|^p dx <\sqrt{\frac{2}{p}}~~\mbox{for}~p\ge 2
\ee
As an application, Ball also derived the sharp constant for cube slicing inequality in \cite{Ball86, Ball89}, which, together with Rogozin's convolution inequality (\cite{Rog87:1}) and a rearrangement argument (\cite{BLL74, WM14}), can also be used to derive the sharp constants for $\infty$-R\'{e}nyi entropy power inequality (the proofs can be found in \cite{BC14, MMX16R, XMM16isita, XMM16isitb}, some basic facts about entropy power inequality can be found in a survey paper \cite{MMX16:1}) of the following form:
\be\label{inq:epireal}
N_\infty(X_1+\cdots +X_n)\ge \frac{1}{2} \sum_{i=1}^nN_\infty(X_i)
\ee
for independent one-dimensional random variables $X_i$, with the notation of $\infty$-R\'enyi entropy power $N_\infty(X):=\|f\|_\infty^{-2}$, where $f$ is the density of $X$. As a further application, Ball's integral inequality \eqref{inq:Ball} also plays a key role in deriving the sharp bounds for marginal densities of product one-dimensional measures (see \cite{LPP15}). 

On the other hand, the ``discrete version'' of Ball's integral inequality, or equivalently the question about an exact upper bound of the $L^p$ norm of the FT of uniform probability mass function supported on the integer interval $\{0,1,\cdots, l-1\}$ (this norm is also called $L^p$ Lebesgue constant) was still open. Note that this FT is precisely the normalized Dirichlet kernel of length $l$ defined by
\ben
D_l(x):=\frac{\sin l\pi x}{l \sin \pi x}\cdot e^{i(l-1)\pi x}
\een
supported on $[-1/2,1/2]$. Before our work, some asymptotic estimates of the $L^p$ Lebesgue constant have been studied by, for example, \cite{AAJRS07, Ash10, Zg59:book}. Specifically, for $p=1$, it is well known that (see \cite{Zg59:book})
\be
\int_{-1/2}^{1/2}|D_l(x)|dx\simeq \frac{4\log l}{\pi^2l}
\ee
For the case that $p>1$, Anderson et. al. in \cite[Lemma 2.1]{AAJRS07} proved an asymptotically sharp estimate:
\begin{align}\label{inq:old}
\int_{-1/2}^{1/2}|D_l(x)|^pdx 
=\frac{\frac{2}{\pi}\int_0^\infty \left|\frac{\sin u}{u}\right|^pdu}{l}+o_p\left(\frac{1}{l}\right)\le  \left(\sqrt{\frac{2}{p}}+o_p(1)\right) \cdot \frac{1}{l}.
\end{align}
This result also gives the connection between $L^p$ Legesgue constant and Ball's integral inequality \eqref{inq:Ball}. However, no exact upper bound can be derived from these results. 

In this paper, we provide the following new result: For $p\ge 2$ and $l\ge 6$, 
\ben
\int_{-1/2}^{1/2}\left|D_l\right|^pdx<\sqrt{\frac{2}{p}}\cdot\sqrt{\frac{1}{l^2-1}}.
\een
it is easy to see that our upper bound coincides with the asymptotic estimation \eqref{inq:old}. 

We would like to mention something about the method in this paper. Motivated by the proof of \cite{Ball86, NP}, the method is basically to compare the distribution functions of $D_l$ and a carefully truncated Gaussian function, and to use a similar argument as in \cite{NP}.

We will organize this paper as follows. 
In section 2, we will provide the proof of our main result (Theorem \ref{thm:DK}). In section 3, we will describe an application of the main result in deriving a new $\infty$-R\'enyi entropy power inequality for integer valued random variables (see Corollary \ref{thm:EPIdiscrGer}).

\section{An exact upper bound on the $L^p$ Lebesgue constant $p\ge 2$.}
\begin{thm}\label{thm:DK}
Let $l\ge 6$ be an integer. Then for $p\ge 2$, the normalized Dirichlet kernel defined by $D_l(x):=\frac{\sin l\pi x}{l \sin \pi x}\cdot e^{i(l-1)\pi x}$ supported on $\left[-\frac{1}{2},\frac{1}{2}\right]$ satisfies the following integral inequality:
\be\label{inq:mailnorm}
\int_{-1/2}^{1/2}\left|D_l(x)\right|^pdx<\sqrt{\frac{2}{p(l^2-1)}}
\ee
\end{thm}

\begin{lem}\label{lem:step1}
Let $l$ be a positive integer, for $x\in [0,1/l]$, 
\be\label{inq:step1}
\frac{\sin l\pi x}{l\sin \pi x}< \exp \left(\frac{-\pi (l^2-1)x^2}{2}\right)
\ee
\end{lem}

\begin{proof}
We have
\begin{align*}
\frac{\sin l\pi x}{l\sin \pi x} &=\frac{\sin l\pi x}{l \pi x}\cdot \frac{\pi x}{\sin \pi x}
=\frac{\prod_{k=1}^\infty\left(1-\frac{l^2x^2}{k^2}\right)}{\prod_{k=1}^\infty\left(1-\frac{x^2}{k^2}\right)}
=\prod_{k=1}^\infty \frac{1-\frac{l^2x^2}{k^2}}{1-\frac{x^2}{k^2}}
\end{align*}
Note that, for $x\in [0,1/l^2]$,
\ben
\exp \left(\frac{-\pi (l^2-1)x^2}{2}\right)\ge \exp \left(\frac{-\pi^2 (l^2-1)x^2}{6}\right)
=\prod_{k=1}^\infty\exp \left(-(l^2-1)\frac{x^2}{k^2}\right)> \prod_{k=1}^\infty \left(1-(l^2-1)\frac{x^2}{k^2}\right)
\een
Compare the right hand sides of these two expressions, it is sufficient to prove the inequality holds termwise, which is:
\be\label{inq:term}
\frac{1-l^2x}{1-x}\le 1-(l^2-1)x,~\mbox{for~}x\in[0,1/l^2]
\ee
which is clearly true.
\end{proof}

In order to prove Theorem \ref{thm:DK}, we will apply \cite[Lemma on distribution functions]{NP}, we state this lemma as follows.

\begin{lem}\label{lem:NP}
For a non-negative function $f: \mathbb{R}\rightarrow [0,\infty)$, its distribution function $F(y)$, $y>0$ is defined by
\ben
F(y):=\lambda \{x\in\mathbb{R}:~f(x)>y\}
\een 
where $\lambda$ is the Lebesgue measure. Let $f$ and $g$ be any two nonnegative measurable functions on $\mathbb{R}$. Let $F$ and $G$ be their distribution functions. Assume that both $F (y)$ and $G(y)$ are finite for every $y > 0$. Assume also that at some point $y_0$ the difference $F-G$ changes sign from $-$ to $+$. Let $S:=\{x>0:~f^p-g^p\in L^1(\mathbb{R})\}$, then the function
\ben
\phi(p):=\frac{1}{py_0^p}\int_\mathbb{R}(f^p-g^p)d\lambda
\een
is increasing on $p$. In particular, if $\int_\mathbb{R}(f^{p_0}-g^{p_0})d\lambda\ge 0$, then $\int_\mathbb{R}(f^p-g^p)d\lambda\ge 0 $
for each $p > p_0$. The equality may hold only if the functions $F$ and $G$ coincide.
\end{lem} 

\begin{proof}[Proof of Theorem \ref{thm:DK}.]
In order to apply Lemma \ref{lem:NP}, we construct our functions $f$ as follows: 
\begin{itemize}
\item If $l$ is even,
\ben
f(x):= \begin{cases}
\exp \left(\frac{-\pi (l^2-1)x^2}{2}\right)~&\mbox{on}~\left[0,\sqrt{\frac{2\log \pi\left(\frac{l}{2} +\frac{1}{2}\right)}{\pi (l^2-1)}}\right]\\
0~&\mbox{otherwise}
\end{cases}
\een
\item If $l$ is odd,
\ben
f(x):= \begin{cases}
\exp \left(\frac{-\pi (l^2-1)x^2}{2}\right)~&\mbox{on}~\left[0,\sqrt{\frac{2\log \pi\left(\lfloor\frac{l}{2}\rfloor +\frac{3}{2}\right)}{\pi (l^2-1)}}\right]\\
0~&\mbox{otherwise}
\end{cases}
\een
\end{itemize}
Note that $f$ is actually the truncated part of the Gaussian function $\exp \left(\frac{-\pi (l^2-1)x^2}{2}\right)$ with the values larger than $\frac{1}{\pi(l/2+1/2)}$ for even $l$ or larger than  $\frac{1}{\pi(\lfloor l/2\rfloor+3/2)}$ for odd $l$. Now define $g(x)$ as follows:
\ben
g(x):=\begin{cases}
\left|\frac{\sin l\pi x}{l\sin \pi x}\right|~&\mbox{on}~[0,1/2] \\
0~&\mbox{otherwise}
\end{cases}
\een

Note that it suffices to prove that $\int_0^\infty g^p\le \int_0^\infty f^p$ for $p\ge 2$, which implies that, for $p\ge 2$, 
\ben
\int g^p\le \int f^p\le \int_\mathbb{R} \exp \left(\frac{-\pi p (l^2-1)x^2}{2}\right) dx=\sqrt{\frac{2}{p}}\cdot\sqrt{\frac{1}{l^2-1}},
\een
which provides the theorem. Note that for $p_0=2$, by Parseval's identity (note that $g$ is actually the absolute value of the FT of the uniform probability distribution on $\{1,2,\cdots, l\}$), we have $\int g^2=1/l$. On the other hand, we claim that $\int f^2\ge 1/l$. In fact, for $l$ even, 
\begin{align*}
\int f^2=\frac{1}{\sqrt{l^2-1}}-2 \int_{\sqrt{\frac{2\log \pi\left(\frac{l}{2} +\frac{1}{2}\right)}{\pi (l^2-1)}}}^\infty e^{-\pi (l^2-1)x^2} dx 
\end{align*}
So it suffices to prove that
\be\label{inq:mess1}
2 \int_{\sqrt{\frac{2\log \pi\left(\frac{l}{2} +\frac{1}{2}\right)}{\pi (l^2-1)}}}^\infty e^{-\pi (l^2-1)x^2} dx \le \frac{1}{\sqrt{l^2-1}}-\frac{1}{l}
= \frac{1}{l\left(l+\sqrt{l^2-1}\right)\sqrt{l^2-1}}
\ee
In fact, the left hand side of \eqref{inq:mess1} has the following estimation:
\begin{align*}
2 \int_{\sqrt{\frac{2\log \pi\left(\frac{l}{2} +\frac{1}{2}\right)}{\pi (l^2-1)}}}^\infty e^{-\pi (l^2-1)x^2} dx
=&\frac{2}{\sqrt{\pi (l^2-1)}} \int_{\sqrt{2\log\pi\left(\frac{l}{2}+\frac{1}{2}\right)}}^\infty e^{-x^2} dx\\
\le &\frac{2}{\sqrt{\pi (l^2-1)}} \int_{\sqrt{2\log\pi\left(\frac{l}{2}+\frac{1}{2}\right)}}^\infty e^{-x^2} x dx\\
=&\frac{1}{\sqrt{\pi (l^2-1)}}  \int_{2\log\pi(\frac{l}{2}+\frac{1}{2})}^\infty e^{-x}dx\\
=& \frac{4}{\pi^{\frac{5}{2}}(l+1)^2\sqrt{(l^2-1)}}
\end{align*}
Comparing this with the right hand side of \eqref{inq:mess1}, it suffices to prove that
\ben
\frac{8}{\pi^{\frac{5}{2}}(l+1)^2}\le \frac{1}{l\left(\frac{l+\sqrt{l^2-1}}{2}\right)},
\een
which is clearly true by the fact that $l+1\ge l\ge\frac{l+\sqrt{l^2-1}}{2}$ and that $\frac{8}{\pi^{\frac{5}{2}}}<1$. For the case that $l$ is odd,  a similar argument will show that $\int f^2\ge 1/l$.

Now it is enough to show that the corresponding distribution functions $F$ and $G$ satisfy the conditions of Lemma \ref{lem:NP} with $s_0=2$. Observe that both $f$ and $g$ are bounded above by 1 by the fact that $|\sin (nx)|\le n|\sin (x)|$.  So we have 
\ben 
F(y)&=&G(y)=0~\mbox{for}~ y\ge 1,\\
F(0)&=&
\begin{cases}
\sqrt{\frac{2\log \pi\left(\frac{l}{2} +\frac{1}{2}\right)}{\pi (l^2-1)}}~~&\mbox{for}~l~\mbox{even}\\
\sqrt{\frac{2\log \pi\left(\lfloor\frac{l}{2}\rfloor +\frac{3}{2}\right)}{\pi (l^2-1)}} ~~&\mbox{for}~l~\mbox{odd}
\end{cases}\\
&<& 1/2=G(0) ~~\mbox{for~all}~l\ge 6.
\een
So we restrict $y\in (0,1)$. Then it is easy to compute the distribution function of $f$ is:
\begin{itemize}
\item If $l$ is even,
\be\label{eq:Fyeven}
F(y)= 
\begin{cases}
\sqrt{\frac{2\log \pi\left(\frac{l}{2} +\frac{1}{2}\right)}{\pi (l^2-1)}} &~~~y\in \left[0,\frac{1}{\pi(l/2+1/2)}\right)\\
\sqrt{\frac{2\log \frac{1}{y}}{\pi (l^2-1)}} & ~~~y\in \left[\frac{1}{\pi(l/2+1/2)}, 1\right]
\end{cases}
\ee
\item If $l$ is odd,
\be\label{eq:Fyodd}
F(y)= 
\begin{cases}
\sqrt{\frac{2\log \pi\left(\lfloor\frac{l}{2}\rfloor +\frac{3}{2}\right)}{\pi (l^2-1)}}  &~~~y\in \left[0,\frac{1}{\pi(\lfloor l/2\rfloor+3/2)}\right)\\
\sqrt{\frac{2\log \frac{1}{y}}{\pi (l^2-1)}} & ~~~y\in \left[\frac{1}{\pi(\lfloor l/2\rfloor+3/2)}, 1\right]
\end{cases}
\ee
\end{itemize}
Now we will estimate $G$. Note that $g(x)$'s graph is like a series of bumps with decreasing heights (see Figure 1 and 2). Consider $y_m:=\max_{[\frac{m}{l},\frac{m+1}{l}]}g$ for $m=\{1,\cdots, l/2-1\}$ for even $l$ or $m= \{1,\cdots ,\lfloor l/2 \rfloor\}$ for odd $l$ (Note that $y_m$ is the peak of each bump). Clearly $y_m\in [\frac{1}{l\sin \pi (m+1/2)/l}, \frac{1}{l\sin \pi m/l}]$. For $x\in [0,1/l]$, by Lemma \ref{lem:step1}, $g(x)<f(x)$, which means that for $y\in (y_1,1)$, $G(y)<F(y)$. Combining this fact with $F(0)<G(0)$, we claim that $F-G$ must change sign at least once on $[0,1]$.  
\begin{figure}\label{fig:1}
\centering
\includegraphics[width=60mm]{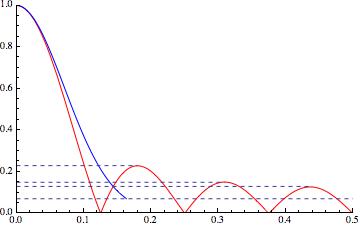}
\caption{Graph of $g(x)$ and $f(x)$ with $l=8$, the red curve is $g(x)$, the blue curve is $f(x)$, and the $y$-values of the dashed horizontal lines are $y_1$, $y_2$, $y_3$, $y_{last}$ from high to low, where $y_{last}$ is the lower bound of positive $f(x)$, see \eqref{eq:deflast}.}
\end{figure}
\begin{figure}\label{fig:2}
\centering
\includegraphics[width=60mm]{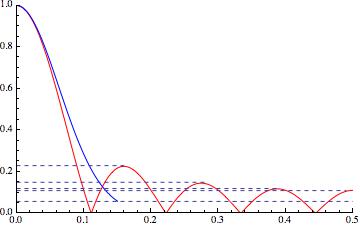}
\caption{Graph of $g(x)$ and $f(x)$ with $l=9$. The $y$-values of the dashed horizontal lines are $y_1$, $y_2$, $y_3$, $y_4$, and $y_{last}$ from high to low}
\end{figure}

To prove that the change of sign of $F-G$ occurs only once, it suffices to prove that $F-G$ is increasing on $(0,y_1)$, which is enough to prove that $|G'(y)|\ge |F'(y)|$ on $(0,y_1)$. Clearly, for each $y\in (0,y_1)$ with $y\neq y_j$, 
\be\label{eq:Gprime}
|G'(y)|=\sum_{x>0:~g(x)=y}\frac{1}{|g'(x)|} 
\ee
when $y\in (y_{m+1},y_m)$, the equation $g(x)=y$ has one root in $(0,1/l)$ and two roots in $(k/l,(k+1)/l)$, $k=1,\cdots, \lfloor l/2\rfloor-1$. In particular, if $l$ is odd and $m=\lfloor l/2\rfloor$, then $g(x)$ has possibly one root in $\left(\frac{\lfloor l/2\rfloor}{l}, \frac{1}{2}\right)$ (see Figure 2 for the case that $l$ is odd). We have, 
\begin{enumerate}
\item If the root $x\in (k/l,(k+1)/l)$ for $k\ge 1$, we claim that
\be\label{inq:mess2}
\left|g'(x)\right|\le \frac{\pi }{\sin\pi x}\cdot \frac{\pi x}{\sin k\pi /l}\le \frac{l\pi^2}{4k}
\ee
In fact, 
\begin{align*}
\left|g'(x)\right|=& \left|\frac{l^2\pi\cos l\pi x\sin\pi x-l\pi \sin l\pi x\cos\pi x}{l^2\sin^2\pi x}\right|\\
=& \left|\frac{\pi }{\sin\pi x}\left(\cos l\pi x-\frac{\cos\pi x\sin l\pi x}{l\sin \pi x}\right)\right|\\
\le &\frac{\pi }{\sin\pi x}\left(1+\frac{\left|\cos\pi x\sin (l\pi x-k\pi)\right|}{l\sin k\pi/l }\right)\\
\le &\frac{\pi }{\sin\pi x}\left(1+\frac{\pi x}{\sin k\pi /l}-\frac{k\pi}{l\sin k\pi/l}\right)\\
\le &\frac{\pi }{\sin\pi x}\cdot \frac{\pi x}{\sin k\pi /l}\le \frac{l\pi^2}{4k}
\end{align*}
where the last step is by the fact that $\sin \pi x\ge 2x$ for $x\in [0,1/2]$. 
\item If the root $x\in (0,1/l)$, we have
\begin{align*}
\left|g'(x)\right|=& \left|\frac{l^2\pi\cos l\pi x\sin\pi x-l\pi \sin l\pi x\cos\pi x}{l^2\sin^2\pi x}\right|\\
= &\frac{l \pi \cos l\pi x\cos\pi x}{\sin^2\pi x}\left|l\tan\pi x-\tan l\pi x\right|
\end{align*}
We claim that, for $x\in (0,1/l)$, one always has
\begin{align*}
\left|g'(x)\right|=&\frac{l \pi \cos l\pi x\cos\pi x}{\sin^2\pi x}\left(\tan l\pi x-l\tan\pi x\right)
\end{align*}
In fact, it is easy to prove that $\tan l\pi x\ge l\tan\pi x$ for $x\in (0,\frac{1}{2l})$. On the other hand, if $x\in (\frac{1}{2l},\frac{1}{l})$, $\tan l\pi x<0$, $l\tan\pi x>0$, but the common factor $\cos l\pi x<0$. By this observation, we have
\begin{align*}
\left|g'(x)\right|=&\frac{l \pi \cos l\pi x\cos\pi x}{\sin^2\pi x}\left(\tan l\pi x-l\tan\pi x\right)\\
=&\frac{\pi\cos\pi x }{l\sin^2\pi x}\left(\sin l\pi x-l\tan\pi x\cos l\pi x\right)\\
\end{align*}
By the fact that $\tan\pi x\ge \pi x$ for $x\in (0,1/l)$, we claim that
\be\label{inq:mess2'}
\left|g'(x)\right|\le \frac{l\pi }{2} \left(\frac{\frac{\pi}{l}}{\sin\frac{\pi}{l}}\right)^2
\ee
In fact, 
\begin{align*}
\left|g'(x)\right|\le &\frac{\pi\cos\pi x }{l\sin^2\pi x}\left(\sin l\pi x-l\pi x\cos l\pi x\right)\\
=& \frac{\pi\cos\pi x }{l\sin^2\pi x}\int_0^{l\pi x}t\sin tdt\\
\le & \frac{\pi }{l\sin^2\pi x}\int_0^{l\pi x}tdt\\
= & \frac{l\pi }{2} \left(\frac{\pi x}{\sin\pi x}\right)^2\\
\le & \frac{l\pi }{2} \left(\frac{\frac{\pi}{l}}{\sin\frac{\pi}{l}}\right)^2\\
\le &\frac{l\pi }{2} \left(\frac{\frac{\pi}{4}}{\sin\frac{\pi}{4}}\right)^2\\
\le & 2l
\end{align*}
for $l\ge 6$. 
\end{enumerate}
Now in order to combine these two cases, we define an extra $y_{last}$ (note that $y_{last}$ is the lower bound of positive $f(x)$, See Figure 1 and 2).
\be\label{eq:deflast}
y_{last}:=
\begin{cases}
\frac{1}{\pi(l/2+1/2)}=:y_{l/2} ~~&\mbox{for}~l~\mbox{even}\\
\frac{1}{\pi(\lfloor l/2\rfloor+3/2)}=:y_{\lceil l/2\rceil} ~~&\mbox{for}~l~\mbox{odd}
\end{cases}
\ee
we have two situations: 
\begin{enumerate}
\item $y\in  \left[0, y_{last}\right]$. For this case, recall \eqref{eq:Fyeven} and \eqref{eq:Fyodd}, $F(y)$ is constant on $ \left[0, y_{last}\right]$, which means that $F'(y)$ is $0$ on this interval. So naturally we have $|G'(y)|\ge |F'(y)|$ on this interval.
\item  $y\in (y_{last},y_1)$. For this case, $y$ must fall into some $(y_{m+1},y_m)$, where $y_{m+1}$ could be $y_{last}$ as in our definition \eqref{eq:deflast}. For every $y\in (y_{m+1},y_m)$, combine \eqref{eq:Gprime}, \eqref{inq:mess2} and \eqref{inq:mess2'},
\be\label{inq:mess3}
|G'(y)|\ge \frac{1}{2l}+2\sum_{k=1}^m\frac{4k}{l\pi^2}-\frac{4m}{l\pi^2}
\ee
where the term with negative sign is to avoid the case that $l$ is odd and $m=\lfloor l/2\rfloor$, where $g(x)$ has only one root in $(\lfloor l/2\rfloor, l/2)$ (see Figure 2). Thus,
\be\label{inq:mess3'}
|G'(y)|\ge \frac{1}{2l}+\frac{4m^2}{l\pi^2}
\ee
On the other hand, we have 
\begin{align*}
\frac{1}{|F'(y)|}=\sqrt{2\pi (l^2-1)\log \frac{1}{y}}\cdot y
\end{align*}
One obtains
\begin{align*}
\left|\frac{G'(y)}{F'(y)}\right|\ge  \left(\frac{1}{2l}+\frac{4m^2}{l\pi^2}\right)y\sqrt{2\pi (l^2-1)\log \frac{1}{y}}
\end{align*}
Note that the function $y\sqrt{\log\frac{1}{y}}$ is increasing on $(0,\frac{1}{\sqrt{e}})$ and decreasing on $(\frac{1}{\sqrt{e}},1)$. Now recall that $y_1\le \frac{1}{l\sin(\pi/l)}\le \frac{1}{4\sin(\pi/4)}=\frac{1}{2\sqrt{2}}<\frac{1}{\sqrt{e}}$, hence $y\sqrt{\log\frac{1}{y}}$ increases on $(0,y_1)$. Moreover, we claim that for $y\in( y_{m+1},y_m)$, one always has $y\ge \frac{1}{\pi (m+\frac{3}{2})}$. In fact, 
\begin{itemize}
\item For the case that $y_{m+1}>y_{last}$, one has $y\ge [\frac{1}{l\sin \pi (m+3/2)/l}\ge \frac{1}{\pi (m+\frac{3}{2})}$. 
\item For the case that $y_{m+1}=y_{last}$, which means that $m=l/2-1$ for $l$ even or $m= \lfloor l/2\rfloor$ for $l$ odd. Thus we surely have $y\ge y_{last}=\frac{1}{\pi (m+\frac{3}{2})}$ by the definition \eqref{eq:deflast}.
\end{itemize}

So we have, for $l\ge 6$,
\begin{align*}
\left|\frac{G'(y)}{F'(y)}\right|
\ge &\left(\frac{2}{\pi}\right)^{5/2}\cdot\frac{\sqrt{l^2-1}}{l}\cdot\frac{\left(m^2+\frac{\pi^2}{8}\right)\sqrt{\log\left(m+\frac{3}{2}\right)\pi}}{m+\frac{3}{2}}\\
\ge & 0.3188\cdot \frac{\left(m^2+\frac{\pi^2}{8}\right)\sqrt{\log\left(m+\frac{3}{2}\right)\pi}}{m+\frac{3}{2}}\\
\end{align*}
which is greater than 1 if $m\ge 3$. 

Now we have only two cases left: $y\in (y_3,y_2)$ or $y\in (y_2, y_1)$. For the case that $y\in (y_3, y_2)$ (i.e. $m=2$), note that if $l\ge 6$, then $g(x)=y$ must have two roots on $(1/l,2/l)$. So we can actually sharpen inequalities \eqref{inq:mess3} and \eqref{inq:mess3'} by:
\begin{align*}
|G'(y)|\ge  \frac{1}{2l}+2\sum_{k=1}^m\frac{4k}{l\pi^2}= \frac{1}{2l}+\frac{4(m^2+m)}{l\pi^2}
\end{align*}
Thus, by repeating the same steps, we obtain
\begin{align*}
\left|\frac{G'(y)}{F'(y)}\right|
\ge & 0.3188\cdot \frac{\left(m^2+m+\frac{\pi^2}{8}\right)\sqrt{\log\left(m+\frac{3}{2}\right)\pi}}{m+\frac{3}{2}}\\
\end{align*}
which is clearly greater than 1 for $m=2$. 

For the case that $y\in (y_2, y_1)$, recall inequalities \eqref{inq:mess2} and \eqref{inq:mess2'} and the fact that $g(x)=y$ has one root in $(0,1/l)$ and two roots in $(1/l,2/l)$, we have
\begin{align*}
|G'(y)|\ge & \frac{2}{l\pi}\left(\frac{\sin\frac{\pi}{l}}{\frac{\pi}{l}}\right)^2+\frac{2\sin\pi x\sin\frac{\pi}{l}}{\pi^2 x}~~~\mbox{note that the}~x~\mbox{in~this~inequality~is~in~}(1/l,2/l)\\
\ge &  \frac{2}{l\pi}\left(\frac{\sin\frac{\pi}{l}}{\frac{\pi}{l}}\right)^2+\frac{\sin\frac{2\pi}{l}\sin\frac{\pi}{l}}{\frac{\pi^2}{l} }~~~\mbox{by~the~fact~that~}\frac{\sin x}{x}~\mbox{is decreasing~for~small}~x\\
= & \frac{1}{l}\left(\frac{\sin\frac{\pi}{l}}{\frac{\pi}{l}}\right)^2\left(\frac{2}{\pi}+2\cos\frac{\pi}{l}\right)
\end{align*}
Thus, recall that $y\ge y_2\ge \frac{1}{\frac{5}{2}\pi}$, and the fact that $y\sqrt{\log\frac{1}{y}}$ increases on $(0,y_1)$, we have
\begin{align*}
\left|\frac{G'(y)}{F'(y)}\right|
\ge &\frac{\frac{1}{l}\left(\frac{\sin\frac{\pi}{l}}{\frac{\pi}{l}}\right)^2\left(\frac{2}{\pi}+2\cos\frac{\pi}{l}\right)}{\frac{5}{2}\pi}\sqrt{2\pi(l^2-1)\log \left(\frac{5}{2}\pi\right)}\\
= &\sqrt{\frac{l^2-1}{l^2}}\left(\frac{\sin\frac{\pi}{l}}{\frac{\pi}{l}}\right)^2\left(\frac{4}{5\pi}+\frac{4}{5}\cos\frac{\pi}{l}\right)\frac{\sqrt{2\pi\log\left(\frac{5}{2}\pi\right)}}{\pi}
\end{align*}
which is greater than 1 if $l\ge 7$. Now for $l=6$, then by the series of inequalities after \eqref{inq:mess2}, we have, for $x\in (1/6,2/6)$, 
\ben
|g'(x)|\le \frac{\pi }{\sin\pi x}\left(1+\cos\frac{\pi}{6}\cdot \left(\frac{\pi x}{\sin \frac{\pi }{6}}-\frac{\pi}{6\sin\frac{ \pi}{6}}\right)\right)
\een
Note that the right hand side is increasing for $x\in (1/6,1/3)$ by computing the derivative. Thus we have
\ben
|g'(x)|\le \frac{\pi }{\sin\pi x}\left(1+\cos\frac{\pi}{6}\cdot \left(\frac{\pi x}{\sin \frac{\pi }{6}}-\frac{\pi}{6\sin\frac{ \pi}{6}}\right)\right)\le 2\pi\left(\frac{1}{\sqrt{3}}+\frac{\pi}{6}\right)
\een
Now by \eqref{inq:mess2'}, we have
\ben
|G'(y)|\ge & \frac{1}{3\pi}\left(\frac{\sin\frac{\pi}{6}}{\frac{\pi}{6}}\right)^2+\frac{1}{\pi\left(\frac{1}{\sqrt{3}}+\frac{\pi}{6}\right)}
\een
Thus, 
\begin{align*}
\left|\frac{G'(y)}{F'(y)}\right|
\ge &\frac{\frac{1}{3\pi}\left(\frac{\sin\frac{\pi}{6}}{\frac{\pi}{6}}\right)^2+\frac{1}{\pi\left(\frac{1}{\sqrt{3}}+\frac{\pi}{6}\right)}}{\frac{5}{2}\pi}\sqrt{2\pi(6^2-1)\log \left(\frac{5}{2}\pi\right)}\approx 1.04598>1
\end{align*}
\end{enumerate}



Now, by applying Lemma \ref{lem:NP}, we have that for any $p\ge 2$, 
\ben
\int_{-1/2}^{1/2} |g(x)|^pdx\le \int_{-1/3}^{1/3} f(x)^pdx<\int_\mathbb{R}\exp \left(\frac{-p\pi (l^2-1)x^2}{2}\right)dx=\sqrt{\frac{2}{p(l^2-1)}}
\een
which provides the theorem.
\end{proof}

\section{An $\infty$-R\'enyi entropy power inequality ($\infty$-EPI) for integer-valued random viables}
\label{sec:disc}
Let us firstly introduce some notations.  Let $X$ be an integer valued random variable with probability mass function $f$, denote by $M(X)=M(f):=\|f\|_\infty$. 

\begin{defn}
Let $X$ be an integer valued random variable with probability mass function $f$. Define the $\infty$-R\'enyi entropy $H_\infty(X)$ by
\ben
H_\infty(X)=H_\infty(f):=-\log \|f\|_\infty=-\log M(f).
\een
Define the $\infty$-R\'enyi entropy power by
\ben
N_\infty(X)=N_\infty(f):=e^{2H_\infty(f)}=M(f)^{-2}
\een
\end{defn}

We would like to introduce our motivation for this section. In \cite{MMX16R, XMM16isitb}, we derived a discrete version of Rogozin's convolution inequality (the continuous Euclidean case can be found in \cite{MMX16R, Rog87:1}). We provide the result as follows.

\begin{thm}[\cite{MMX16R, Rog87:1}]\label{cor:RgzItg}
Let $X_1$, $\cdots$, $X_n$ be independent integer valued random variables with $M(X_i)\in \left( \frac{1}{l_i+1}, \frac{1}{l_i} \right]$ for some positive integer $l_i$, then 
\be\label{inq:RgzItg}
M(X_1+\cdots+X_n)\le M(U_1+\cdots+U_n),
\ee
where $U_i$'s are independent integer valued random variables uniformly supported on $\{1,2,\cdots, l_i\}$.
\end{thm}

This result enables us to reduce the estimation of $M(X_1+\cdots+X_n)$ to a discrete cube slicing problem. In particular, if $l_i$'s are the same (i.e. $M(X_i)$'s are not far from each other), then the corresponding $U_i$'s in \eqref{inq:RgzItg} are i.i.d random variables uniformly distributed on $\{1,2,\cdots, l_i\}$. For this special case, a direct result by Mattner and Roos in \cite[Theorem]{MR08} (which proved a sharp upper bound of $M(f^{*n})$ for $f$ uniform probability mass function on a discrete interval) can be applied, which yields the following partial result.

\begin{thm}[\cite{MMX16R, Rog87:1}]\label{thm:EPIdiscr}
For independent integer valued random variables $X_1$, $\cdots$, $X_n$ with $M(X_i)\in \left( \frac{1}{l+1}, \frac{1}{l} \right]$ for some fixed integer $l\ge 2$, 
\be\label{inq:EPIdiscr}
N_\infty\left(\sum_{i=1}^nX_i\right)\ge \frac{\pi}{6} \frac{l^2-1}{(l+1)^2}\sum_{i=1}^nN_\infty(X_i).
\ee
In particular, if all $M(X_i)=l$, 
\be\label{inq:EPIdiscr2}
N_\infty\left(\sum_{i=1}^nX_i\right)\ge \frac{\pi}{6} \frac{l^2-1}{l^2}\sum_{i=1}^nN_\infty(X_i).
\ee
\end{thm}

\begin{rmk}
The constants in Theorem \ref{thm:EPIdiscr} are asymptotically sharp as $n\rightarrow\infty$ and $l\rightarrow\infty$. In fact, as $l$ large enough, then the constant $\approx \frac{\pi}{6}$, which is the optimal constant by local central limit theorem. 
\end{rmk}

However, for the case that $M(X_i)$'s are far from each other, the argument of Mattner and Roos fails to apply. We will have to use our main result \eqref{inq:mailnorm}. We state this new $\infty$-EPI as follows. 

\begin{cor}\label{thm:EPIdiscrGer}
For independent integer valued random variables $X_1$, $\cdots$, $X_n$ with $M(X_i)\in \left( \frac{1}{l_i+1}, \frac{1}{l_i} \right]$ for some integers $l_i$, denote $l_{min}:=\min_i l_i$ and $l_{max}:=\max_i l_i$, and assume that $l_{min}\ge 6$, then the following $\infty$-EPI holds
\be\label{inq:maingel}
N_\infty\left(\sum_{i=1}^nX_i\right)\ge \frac{1}{2}\cdot \frac{l_{min}-1}{l_{min}+1} \sum_{i=1}^nN_\infty(X_i)\ge \frac{5}{14}\sum_{i=1}^nN_\infty(X_i)
\ee
In particular, if $M(X_i)=1/l_i$, 
\be\label{inq:maingelhit}
N_\infty\left(\sum_{i=1}^nX_i\right)\ge\frac{1}{2}\cdot \frac{l_{min}^2-1}{l_{min}^2} \sum_{i=1}^n N_\infty(X_i)\ge \frac{35}{72} \sum_{i=1}^n N_\infty(X_i)
\ee
\end{cor}

\begin{proof}
By Theorem \ref{cor:RgzItg}, we have $M(X_1+\cdots+X_n)\le M(U_1+\cdots+U_n)$, where $U_i$'s are independent integer valued random variables uniformly supported on $\{1,2,\cdots, l_i\}$. Let $g_i$ be the probability mass functions of $U_i$, thus by Hausdorff-Young inequality for discrete groups, we have
\begin{align*}
 M(U_1+\cdots+U_n)=M(g_i*\cdots * g_n)\le \left\|\prod_{i}D_{l_i}\right\|_1
\end{align*}
Now we have two cases:
\begin{enumerate}
\item Case 1: $\frac{l_{max}^2}{\sum_j l_j^2}\le 1/2$
\item Case 2: $\frac{l_{max}^2}{\sum_j l_j^2}> 1/2$
\end{enumerate}

For Case 1, let $p_i:=\frac{\sum_j l_j^2}{l_i^2}$, then clearly $p_i\ge 2$ and $\sum_i \frac{1}{p_i}=1$, then by H\"{o}lder's inequality and Theorem \ref{thm:DK}, 
\begin{align*}
\left\|\prod_{i}D_{l_i}\right\|_1^2
\le & \prod_{i=1}^n\|D_{l_i}\|_{p_i}^2\\
\le &\prod_{i=1}^n \left(\frac{2}{p_i(l_i^2-1)}\right)^{\frac{1}{p_i}}\\
\le &\frac{2l_{min}^2}{l_{min}^2-1}\cdot\frac{1}{\sum_{i=1}^nl_i^2}
\end{align*}
which is exactly \eqref{inq:maingelhit} for this case. Furthermore, we have
\begin{align*}
N_\infty\left(\sum_{i=1}^nX_i\right)
\ge &\frac{1}{2}\cdot \frac{l_{min}^2-1}{l_{min}^2} \sum_{i=1}^n N_\infty(U_i)\\
\ge &\frac{1}{2}\cdot \frac{l_{min}^2-1}{l_{min}^2} \sum_{i=1}^n \frac{l_i^2}{(l_i+1)^2}N_\infty(X_i)\\
\ge &\frac{1}{2}\cdot \frac{l_{min}-1}{l_{min}+1} \sum_{i=1}^nN_\infty(X_i)
\end{align*}
which provides inequality \eqref{inq:maingel} for this case. 

For Case 2, by the fact that $N_\infty\left(\sum_{i=1}^nU_i\right)\ge N_\infty (U_j)$ for each $j$, 
\ben
N_\infty\left(\sum_{i=1}^nU_i\right)\ge l_{max}^2>\frac{1}{2}\sum_{i=1}^n N_\infty(U_i)
\een
which provides inequality \eqref{inq:maingelhit} for this case. Furthermore, 
\begin{align*}
N_\infty\left(\sum_{i=1}^nX_i\right)
\ge & N_\infty\left(\sum_{i=1}^nU_i\right)>\frac{1}{2}\sum_{i=1}^n N_\infty(U_i)\ge \frac{1}{2}\sum_{i=1}^n \frac{l_i^2}{(l_i+1)^2}N_\infty(X_i)\\
\ge &\frac{1}{2}\cdot \frac{l_{min}^2}{(l_{min}+1)^2}\sum_{i=1}^nN_\infty(X_i)\ge \frac{1}{2}\cdot \frac{l_{min}-1}{l_{min}+1}\sum_{i=1}^nN_\infty(X_i)
\end{align*}
which provides inequality \eqref{inq:maingel} for this case. 
\end{proof}

\begin{rmk}
In Corollary \ref{thm:EPIdiscrGer}, it is easy to see that as $l_{min}\rightarrow\infty$, the constant is asymptotically $\frac{1}{2}$, which is asymptotically sharp in the sense that $N_\infty(X+X')=N(X)$ for the case that $X$ is uniformly distributed on a discrete interval and $X'$ is the independent copy of $X$.
\end{rmk}

\bibliographystyle{elsarticle-num}

\begin{thebibliography}{10}
\small

\bibitem{AAJRS07}
Anderson, B., Ash, J.M., Jones, R., Rider, D.G., Saffari, B.: Exponential sums with coefficients 0 or 1
and concentrated $L^p$ norms. 
{\em Ann. Inst. Fourier} 57, 1377?1404 (2007)

\bibitem{Ash10}
Marshall Ash, Triangular Dirichlet Kernels and Growth of $L^p$ Lebesgue Constants
{\em J Fourier Anal Appl} (2010) 16: 1053?1069

\bibitem{Ball86}
K. Ball. Cube slicing in $\mathbb{R}^n$. 
{\em Proc. Amer. Math. Soc.}, 97(3):465-473, 1986.

\bibitem{Ball89}
K. Ball. Volumes of sections of cubes and related problems. 
In {\em Geometric aspects
of functional analysis} (1987-88), volume 1376 of {\em Lecture Notes in Math.}, pp. 251-
260. {\em Springer}, Berlin, 1989.

\bibitem{Bec75}
W. Beckner. Inequalities in Fourier analysis. 
{\em Ann. of Math.} (2), 102(1):159?182, 1975.

\bibitem{BC14}
S. G. Bobkov and G. P. Chistyakov.
Bounds for the maximum of the density of the sum of independent random variables. 
{\em Zap. Nauchn. Sem. S.-Peterburg. Otdel. Mat. Inst.
Steklov. (POMI)}, 408(Veroyatnost i Statistika. 18):62-73, 324, 2012.

\bibitem{BLL74}
H. J. Brascamp, E. H. Lieb, and J. M. Luttinger. A general rearrangement inequality
for multiple integrals. {\em J. Functional Analysis}, 17:227-237, 1974.

\bibitem{GR10}
J. Gilbert and Z. Rzeszotnik. The norm of the Fourier transform on finite abelian groups. 
{\em Ann. Inst. Fourier (Grenoble)}, 60(4):1317?1346, 2010.

\bibitem{Gf14:book} 
L. Grafakos. Classical Fourier analysis, volume 249 of Graduate Texts in Mathematics. 
{\em Springer},
New York, third edition, 2014.

\bibitem{KM40}
M. Krein and D. Milman. On extreme points of regular convex sets. 
{\em Studia Math.},
9:133-138, 1940.

\bibitem{LPP15}
 G. Livshyts, G. Paouris, and P. Pivovarov. On sharp bounds for marginal densities of product measures. 
 {\em Preprint}, arXiv:1507.07949, 2015.
 
 \bibitem{MR08}
L. Mattner and B. Roos. Maximal probabilities of convolution powers of discrete uniform distributions. 
{\em Statist. Probab. Lett.}, 78(17):2992-2996, 2008.

\bibitem{MMX16:1}
M.~Madiman, J.~Melbourne, and P.~Xu.
\newblock Forward and reverse entropy power inequalities in convex geometry.
\newblock
To appear in: {\em Probability, Convexity and Discrete Analysis}, Volume commemorating
the 2014-15 Annual Program on Discrete Structures at the Institute for Mathematics
and its Applications, ed. E. Carlen, M. Madiman and E. Werner, to be published by
{\em Springer} in 2017. {\tt arXiv:1604.04225}.


\bibitem{MMX16R}
M.~Madiman, J.~Melbourne, and P.~Xu.
\newblock
Rogozin's convolution inequality for locally compact groups.
\newblock {\em Preprint}.

\bibitem{MX16FN}
M. Madiman and P. Xu. 
\newblock
The norm of the Fourier transform on compact or discrete abelian
groups. Submitted to
\newblock{\em
Journal of Fourier Analysis and Applications}, {\tt arXiv:1611.04692}.

\bibitem{MWW16}
M.~Madiman, L. ~Wang, J. O. Woo. 
On Entropy Inequalities of Sums in Prime Cyclic Groups and Their Applications. 
{\em Prinprint}. 2017.

\bibitem{MR08}
L. Mattner and B. Roos. Maximal probabilities of convolution powers of discrete uniform distributions. 
{\em Statist. Probab. Lett.}, 78(17):2992-2996, 2008.
{\em Prinprint}. 2017.

\bibitem{NP}
Fedor L. Nazarov and Anatoliy N. Podkorytov, \newblock
Ball, Haagerup, and distribution functions. 
{\em Complex Analysis, Operators, and Related Topics. Operator Theory: Advances and Applications}, vol 113, pp. 247-267. {\em Birkh\"{a}user}, Basel, 2000.


\bibitem{Rog87:1}
B. A. Rogozin. An estimate for the maximum of the convolution of bounded densities.
{\em Teor. Veroyatnost. i Primenen.}, 32(1):53-61, 1987.

\bibitem{Zg59:book}
A. Zygmund. Trigonometric series. 2nd ed. Vols. I, II. 
{\em Cambridge University Press}, New York,
1959.

\bibitem{XM15}
P. Xu, M. Madiman, The norm of the Fourier series operator. {\em 2015 IEEE International
Symposium on Information Theory (ISIT)}, 750-754

\bibitem{XMM16isita}
P. Xu, J. Melbourne and M. Madiman, Infinity-Renyi Entropy Power Inequalities, accepted by {\em 2017 IEEE International
Symposium on Information Theory (ISIT)}

\bibitem{XMM16isitb}
P. Xu, J. Melbourne and M. Madiman, 
A min-entropy power inequality for groups, accepted by {\em 2017 IEEE International
Symposium on Information Theory (ISIT)}



\bibitem{WM14}
L. Wang and M. Madiman. Beyond the entropy power inequality, via rearrangements.
{\em IEEE Trans. Inform. Theory}, 60(9):5116-5137, September 2014.

%
%
%
%
%
%
%
%
%
%
%
%
%
%
%
%
%
%
%
%
%
%
%
%
%
%
%
%
%
%
%
%
%
%
%
%
%
%
%
%
%
%
%
%
%
%
%
%



\end{thebibliography}


\end{document}